\documentclass[reqno]{amsart}
 \newtheorem{thm}{Theorem}[section]
 \newtheorem{cor}[thm]{Corollary}
 
 \newtheorem{prop}[thm]{Proposition}
 \theoremstyle{definition}
 \newtheorem{defn}[thm]{Definition}
 \theoremstyle{remark}
 \newtheorem{rem}[thm]{Remark}
 
 \numberwithin{equation}{section}

\usepackage{verbatim}
\usepackage{xcolor}
\usepackage{graphicx}
\usepackage{amssymb, amsmath, amsthm}
\usepackage{amsfonts}
\usepackage{amssymb}
\usepackage{float}
\usepackage{hyperref}
\usepackage{mathrsfs}
\usepackage{enumitem}

\newcommand{\dis}{\displaystyle}

\newcommand{\dt}{\partial_{t}}

\newcommand{\0}{\mathbf{0}}

\newcommand{\A}{{\mathcal A}}
\newcommand{\Gs}{G^{s}}

\newcommand{\cB}{\mathcal B}

\newcommand{\bk}{\textbf{\textit{k}}}
\newcommand{\bj}{\textbf{\textit{j}}}
\newcommand{\bl}{\textbf{\textit{l}}}
\newcommand{\bx}{\textbf{\textit{x}}}
\newcommand{\bb}{\textbf{\textit{b}}}
\newcommand{\bbj}{\bb^{(j)}}

\newcommand{\N}{\mathbb N}

\newcommand{\T}{\mathbb T}
\newcommand{\Z}{\mathbb Z}

\newcommand{\R}{\mathbb R}

\newcommand{\st}{\tilde{\theta}_0}

\newcommand{\bt}{\Theta}
\newcommand{\summ}{\sum_{\bj+\bk=\bl;\bj,\bk,\bl\in\mathbb{Z}^{d}_*}}
\newcommand{\suma}{\sum_{(\bj,\bk,\bl)\in\A}}
\newcommand{\sumb}{\sum_{(\bj,\bk,\bl)\in\cB}}

\numberwithin{equation}{section}
\numberwithin{theorem}{section}

\numberwithin{figure}{section}

\begin{document}

%
%
%
%
%
%
%
%
%

\title[Ill/well-posedness of active scalar equations]
 {Ill/well-posedness of non-diffusive active scalar equations with physical applications}

\author{Susan Friedlander}

\address{Department of Mathematics\\
University of Southern California}

\email{susanfri@usc.edu}

\author{Anthony Suen}

\address{Department of Mathematics and Information Technology\\
The Education University of Hong Kong}

\email{acksuen@eduhk.hk}

\author{Fei Wang}

\address{School of Mathematical Sciences, CMA-Shanghai \\
Shanghai Jiao Tong University}

\email{fwang256@sjtu.edu.cn}


\date{}



\begin{abstract}
We consider a general class of non-diffusive active scalar equations with constitutive laws obtained via an operator $\mathbf{T}$ that is singular of order $r_0\in[0,2]$. For $r_0\in(0,1]$ we prove well-posedness in Gevrey spaces $G^s$ with $s\in[1,\frac{1}{r_0})$, while for $r_0\in[1,2]$ and further conditions on $\mathbf{T}$ we prove ill-posedness in $G^s$ for suitable $s$. We then apply the ill/well-posedness results to several specific non-diffusive active scalar equations including the magnetogeostrophic equation, the incompressible porous media equation and the singular incompressible porous media equation.
\end{abstract}

\maketitle
\section{Introduction}\label{introduction}

Active scalar equations belong to a class of partial differential equations where the evolution in time of a scalar quantity is governed by the motion of the fluid in which the velocity itself varies with this scalar quantity. This kind of differential equations comes from numerous physical models such as the Navier-Stokes, Euler, or magneto-hydrodynamic (MHD) equations, all of them have great importance in practical applications ranging from fluid mechanics to atmospheric science, oceanography and geophysics; see for example \cite{CCCGW12, CFG11, CGO07, FGSV12, FRV12, FRV14, FS15, FS18, FS19, FS21b, FV11a, FV11b, KV09, KVW16} and the references therein. They have been a topic of considerable study in recent years, in part because they arise in many physical models and in part because they present challenging nonlinear PDEs. The physics of an active scalar equation is encoded in the constitutive law that relates the transport velocity vector $u$ with a scalar field $\theta$. This law produces a differential operator that when applied to the scalar field determines the velocity. The singular or smoothing properties of this operator are closely connected with the mathematics of the nonlinear advection equation for $\theta$. From the mathematical point of view, these active scalar equations preserve some essential difficulties of the full fluid equations, which give rise to interesting and challenging problems. In this present paper we study issues of ill/well-posedness in Gevrey spaces when the nonlinear term is singular and there is no compensating effect of diffusion. More precisely, we consider a general class of non-diffusive active scalar equations with constitutive laws obtained via an operator $\mathbf{T}$ that is singular of order $r_0\in[0,2]$. For $r_0\in(0,1]$ we prove well-posedness in Gevrey spaces $G^s$ with $s\in[1,\frac{1}{r_0})$, while for $r_0\in[1,2]$ and further conditions on $\mathbf{T}$ we prove ill-posedness in $G^s$ for some suitable values of $s$. Our main results are given in Theorem~\ref{local existence Gevrey thm general} and Theorem~\ref{local ill-posed Gevrey thm general}.

We address the following abstract class of active equations in $\T^d\times[0,\infty)=[0,2\pi]^d\times(0,\infty)$, where $d\ge2$ is the spatial dimension:
\begin{align}
\label{general abstract active scalar eqn nondiffusive} 
\left\{ \begin{array}{l}
\partial_t\theta+u\cdot\nabla\theta=0, \\
\nabla\cdot u=0,\,\,\,u=\mathbf{T}\theta,\,\,\,\theta(\bx,0)=\theta_0(\bx).
\end{array}\right.
\end{align}
The $d$-dimensional singular vector field $u$ is obtained from the unknown scalar field $\theta$ via the Fourier multiplier operator $\mathbf{T}$, which is given explicitly in term of the Fourier symbol $$\widehat{\mathbf{T}} = (\widehat{\mathbf{T}}_1,\dots,\widehat{\mathbf{T}}_{d-1},\widehat{\mathbf{T}}_{d})$$ with $\widehat{\mathbf{T}}:\mathbb{Z}^d\to\R^d$. We also write $\bk'$ to denote the $d-1$ dimensional vector $(k_1, . . . , k_{d-1}) \in \mathbb{Z}^{d-1}$ and we assume the following condition hold:
\begin{align}\label{singular operator condition}
\mbox{$|\widehat{\mathbf{T}}(\bk)|\le C|\bk|^{r_{0}}$ for all $\bk\in\mathbb{Z}^d$,}
\end{align}
where $r_{0}\in[0,2]$ and $C>0$. Without loss of generality we assume that $\dis\int_{\T^d}\theta(\bx,t)d\bx=0$ for all $t\ge0$, since the mean of $\theta$ is conserved by the flow. 

Our motivation for addressing such a class of active scalar equations mainly comes from two different physical systems. The first example comes from MHD and a physical model proposed by Moffatt and Loper \cite{ML94, M08} for magenetostrophic turbulence in the Earth's fluid core. Under the postulates in \cite{ML94}, the governing equation reduces to a three dimensional active scalar equation in $\mathbb{T}^3\times(0,\infty)$ for a temperature field $\theta$, which is given by
\begin{align}
\label{abstract active scalar eqn MG intro}  
\left\{ \begin{array}{l}
\partial_t\theta+u\cdot\nabla\theta=0 \\
u=\mathbf{M}\theta,\,\,\,\theta(\bx,0)=\theta_0(\bx).
\end{array}\right.
\end{align}
Here $\mathbf{M}=(\mathbf{M}_1,\mathbf{M}_2,\mathbf{M}_3)$ is a Fourier multiplier operator with symbols $\mathbf{M}_1$, $\mathbf{M}_2$, $\mathbf{M}_3$ given explicitly by
\begin{align}
\widehat{\mathbf{M}}_1(\bk)&=\frac{2\Omega k_2k_3|\bk|^2-(\beta^2/\eta)k_1k_2^2k_3}{4\Omega^2k_3^2|\bk|^2+(\beta^2/\eta)^2k_2^4},\label{MG Fourier symbol_1}\\
\widehat{\mathbf{M}}_2(\bk)&=\frac{-2\Omega k_1k_3|\bk|^2-(\beta^2/\eta)k_2^3k_3}{4\Omega^2k_3^2|\bk|^2+(\beta^2/\eta)^2k_2^4},\label{MG Fourier symbol_2}\\
\widehat{\mathbf{M}}_3(\bk)&=\frac{(\beta^2/\eta)k_2^2(k_1^2+k_2^2)}{4\Omega^2k_3^2|\bk|^2+(\beta^2/\eta)^2k_2^4},\label{MG Fourier symbol_3}
\end{align}
where $\beta$, $\eta$, $\Omega$ are some fixed positive constants. The expressions \eqref{MG Fourier symbol_1},  \eqref{MG Fourier symbol_2} and \eqref{MG Fourier symbol_3} hold for all $\bk\in\{(k_1,k_2,k_3)\in\mathbb{Z}^3:k_3\neq0\}$, and for the self-consistency of the model we take 
\begin{align*}
\widehat{\mathbf{M}}(k_1,k_2,0) = (0,0,0),\qquad k_1,k_2\in\mathbb{Z}.
\end{align*}
The system \eqref{abstract active scalar eqn MG intro}-\eqref{MG Fourier symbol_3} is known as the non-diffusive magenetostrophic (MG) equation. We refer the readers to~\cite{FV11a, FV11b, FRV14, FS15, FS18, FS21b} for more results concerning the MG equation.

The second physical example comes from a two dimensional model derived via Darcy’s law for the evolution of a flow in a porous medium, which is given by an active scalar equation in $\mathbb{T}^2\times(0,\infty)$
\begin{align}
\label{abstract active scalar eqn IPM intro}  
\left\{ \begin{array}{l}
\partial_t\theta+u\cdot\nabla\theta=0 \\
u=\mathbf{I}\theta,\,\,\,\theta(\bx,0)=\theta_0(\bx).
\end{array}\right.
\end{align}
Here $\mathbf{I}=(\mathbf{I}_1,\mathbf{I}_2)$ is a Fourier multiplier operator with symbols $\mathbf{I}_1$, $\mathbf{I}_2$ given explicitly by
\begin{align}
\widehat{\mathbf{I}}_1(\bk)&=\frac{-k_1k_2}{k_1^2+k_2^2},\label{IPM Fourier symbol_1}\\
\widehat{\mathbf{I}}_2(\bk)&=\frac{k_1^2}{k_1^2+k_2^2},\label{IPM Fourier symbol_2}
\end{align}
The expressions \eqref{IPM Fourier symbol_1} and \eqref{IPM Fourier symbol_2} hold for all $\bk\in\mathbb{Z}^2_{*}$, and for the self-consistency of the model we take 
\begin{align*}
\widehat{\mathbf{I}}(0,0) = (0,0).
\end{align*}
The system \eqref{abstract active scalar eqn IPM intro}-\eqref{IPM Fourier symbol_2} is called the incompressible porous media (IPM) equation. We refer the readers to~\cite{CFG11, CGO07, FGSV12, FS19} for more results on the IPM equation and the singular incompressible porous media (SIPM) equation.

The purpose of our current paper is to address the ill/well-posedness of the abstract class of active scalar equations \eqref{general abstract active scalar eqn nondiffusive} under particular regimes for $r_0$ defined in \eqref{singular operator condition}. The paper is organised as follows. In Section~\ref{wellposed in Gevrey class general} we prove local well-posedness in Gevrey classes for \eqref{general abstract active scalar eqn nondiffusive}. For the special case when the symbol of $\mathbf{T}$ is {\it odd}, the authors in \cite{CCCGW12} showed the local existence and uniqueness of solutions for \eqref{general abstract active scalar eqn nondiffusive} in Sobolev spaces $H^s$, the main ingredients in their proof were an estimate for the commutator $\|[\partial_{x_i}(-\Delta)^\frac{s}{2},g]\|_{L^2}$ and the identity
\begin{align*}
\int_{\R^d}fAfg=-\frac{1}{2}\int_{\R^d}f[A,g]f,
\end{align*}
which holds for smooth $f$ and $g$, and $A$ is an odd operator. In contrast, when the symbol of $\mathbf{T}$ is {\it even} and singular of order 1, the above method fails and in particular for the case of non-diffusive MG equation described by \eqref{abstract active scalar eqn MG intro}, Friedlander and Vicol \cite{FV11b} established ill-posedness in the sense of Hadamard in Sobolev spaces. The proof of ill-posednes in \cite{FV11b} relied on the explicit structure of the symbol of $\mathbf{M}$ given by \eqref{MG Fourier symbol_1}-\eqref{MG Fourier symbol_3} in which there is one derivative loss in $x$ in the map $\theta\mapsto u$, with the anisotropy of the Fourier symbol of $\mathbf{M}$ and the fact that the symbol of $\mathbf{M}$ is even. In our current work, we point out that our well-posedness results in Gevrey classes hold for {\it any} abstract operator $\mathbf{T}$ satisfying \eqref{singular operator condition} and is {\it independent of} the evenness or oddness of $\mathbf{T}$.

In Section~\ref{illposed in Gevrey class general} under further conditions (C1)--(C6) on the abstract system \eqref{general abstract active scalar eqn nondiffusive}, we prove ill-posedness results in Gevrey classes. For even operators $\mathbf{T}$, ill-posedness results in Sobolev spaces have been proved using techniques of continued fractions \cite{FGSV12, FRV12, FV11b}, and such method was later adopted for odd operators as well \cite{KVW16}. In this current paper we extend such methods to Gevrey spaces. The main ingredients in the proof involve a linear ill-posedness result and a classical linear implies nonlinear ill-posedness argument. The conditions (C1)--(C6) described in Section~\ref{illposed in Gevrey class general} are somewhat more general than those given in \cite{FGSV12} as we discuss in Remark~\ref{discussion on the conditions C1 to C6}. Our conditions allow us to obtain sharper results for some physical models. To achieve our goal, we {\it combine the techniques} used for even operators \cite{FGSV12, FRV12, FV11b} as well as odd operators \cite{KVW16}, which enable us to obtain a fast enough growth rate of the $L^2$ norm of the solutions in any short period of time. 

In Section~\ref{Applications to physical models section} we apply the results for a general class of active scalar equations to the MG, IPM and SIPM equations. Our end point results are consistent with the literature and give a more complete picture of Gevrey space ill/well-posedness related to strongly singular operators. We point out that our ill-posedness results apply to the physical model of non-diffusive MG equation given by \eqref{abstract active scalar eqn MG intro}-\eqref{MG Fourier symbol_3}, which show that the non-diffusive MG equation is in fact locally ill-posed in Gevrey spaces $G^s$ {\it for all} $s>1$. Such result is drastically different from the case of analytic spaces (that is the case of $G^s$ with $s=1$) for which the authors in \cite{FV11b} proved that the non-diffusive MG equation is locally well-posed in analytic spaces. If the operator $\mathbf{M}$ in \eqref{abstract active scalar eqn MG intro} is replaced by $(-\Delta)^\frac{\alpha}{2}\mathbf{M}$, then we are able to obtain a {\it sharp dichotomy} across the value $\alpha=0$ that is the case of non-diffusive MG equation. More precisely, if $\alpha<0$ the equations are locally well-posed, while if $\alpha>0$ they are ill-posed in Gervey classes $G^s$ for all $s\ge1$. We refer to Remark~\ref{discussion on sharp dichotomy} for further discussion.  

\section{Preliminaries}\label{Preliminaries}

We introduce the following notations used in this paper. $W^{s,p}$ is the usual inhomogeneous Sobolev space with norm $\|\cdot\|_{W^{s,p}(\T^{d})}$, and we write $H^s=W^{s,2}$. For simplicity, we write $\|\cdot\|_{L^p}=\|\cdot\|_{L^p(\T^{d})}$, $\|\cdot\|_{W^{s,p}}=\|\cdot\|_{W^{s,p}(\T^{d})}$, etc. unless otherwise specified.  

We recall the following Sobolev embedding inequality from the literature (see for example Bahouri-Chemin-Danchin \cite{BCD11} and Ziemer \cite{Z89}): For $d\ge1$, suppose that $k>0$, $p<d$, $1\le p<q<\infty$ with $\frac{1}{p} - \frac{l}{d}=\frac{1}{q}$. If $g\in W^{p,l}$, then $g\in L^q$ and there exists a constant $C>0$ such that
\begin{align}\label{Sobolev inequality}
\|g\|_{L^q}\le C\|g\|_{W^{p,l}}.
\end{align}

We state the following definition for Gevrey class $\Gs$ when $s\ge1$.

\begin{defn}[Gevrey-space]
Fix $r>3$. A function $\theta\in C^{\infty}(\R^d)$ belongs to the Gevrey class $\Gs$ where $s\ge1$, if there exists $\tau>0$, known as the Gevrey-class radius, such that the $\Gs_\tau$-norm is finite, i.e.
\begin{align}\label{defn of Gevrey norm}
\|\theta\|^2_{\Gs_\tau}=\|(-\Delta)^\frac{r}{2} e^{\tau(-\Delta)^\frac{1}{2s}}\theta\|^2_{L^2}=\sum_{\bk\in\Z^{d}_{*}}|\bk|^{2r}e^{2\tau|\bk|^\frac{1}{s}}|\widehat{\theta}(\bk)|^2.
\end{align}
\end{defn}
\begin{rem}
Since we are working in the mean-free setting, we take $\Z^{d}_{*}=\{\bk\in\Z^{d}:|\bk|\neq0\}$ and we define $\Gs:=\cup_{\tau>0}\Gs_\tau$. We point out that for the case when $s=1$, $\Gs$ gives the space of analytic functions.
\end{rem}

We also recall the following definition of Lipschitz well-posedness for the nonlinear equation (see \cite[Definition 4.4]{FRV12} for example):

\begin{defn}\label{def of locally Lipschitz wellposedness}
Let $Y\subset X\subset H^{q}$ be Banach spaces with $q>r_{0}+\frac{d}{4}$. The Cauchy problem \eqref{general abstract active scalar eqn nondiffusive} is locally Lipschitz $(X, Y )$ well-posed, if there exist continuous functions $T, K:[0, \infty)^2\to(0, \infty)$, the time of existence and the Lipschitz constant, so that for every pair of initial data $\theta_0^{(1)}$, $\theta_0^{(2)}\in Y$, there exist unique solutions $\theta^{(1)}$, $\theta^{(2)}\in L^\infty(0,T;X)$ of the initial value problem associated with \eqref{general abstract active scalar eqn nondiffusive} that additionally satisfy
\begin{align}\label{condition for solution being Lipschitz wellposed}
\|\theta^{(1)}(\cdot,t)-\theta^{(2)}(\cdot,t)\|_{X}\le K\|\theta_0^{(1)}(\cdot)-\theta_0^{(2)}(\cdot)\|_{Y}
\end{align}
for every $t\in [0,T]$, where $T=T(\|\theta_0^{(1)}\|_{Y},\|\theta_0^{(2)}\|_{Y})$ and $K=K(\|\theta_0^{(1)}\|_{Y},\|\theta_0^{(2)}\|_{Y})$.
\end{defn}

Finally, we state the following fact which will become useful in Section~\ref{wellposed in Gevrey class general}.

\begin{prop}\label{useful algebraic ineq prop}
Let $\zeta>0$. Then for any $a,b>0$, we have
\begin{align}\label{useful algebraic ineq}
(a+b)^\zeta\le C_{\zeta}(a^\zeta + b^\zeta)
\end{align}
where $C_{\zeta}>0$ depends only on $\zeta$ and is independent of $a$ and $b$.
\end{prop}

\section{Wellposedness in Gevrey-class of active scalar equations}\label{wellposed in Gevrey class general}

In this section, we study the active scalar equation \eqref{general abstract active scalar eqn nondiffusive} when $r_0\in[0,1]$, where $r_0$ is given by \eqref{singular operator condition} in Section~\ref{introduction}. We prove that the equation \eqref{general abstract active scalar eqn nondiffusive} is locally well-posed in Gevrey class $\Gs$ for all $s\in[1,\frac{1}{r_{0}}]$ when $r_0\in(0,1]$, while for the case $r_0=0$, the equation \eqref{general abstract active scalar eqn nondiffusive} is locally well-posed in $\Gs$ for all $s\ge1$. The results are summarised in Thereom~\ref{local existence Gevrey thm general}:

\begin{thm}\label{local existence Gevrey thm general}
Let $r_{0}\in[0,1]$ and assume that \eqref{singular operator condition} holds. When $r_{0}\in(0,1]$, we fix $s\in[1,\frac{1}{r_{0}}]$ and $K_0>0$. Let $\theta_0\in \Gs$ with radius of convergence $\tau_0>0$ and
\begin{align}\label{bounds on gevrey norm of theta0}
\|(-\Delta)^\frac{r}{2}e^{\tau_0(-\Delta)^\frac{1}{2s}}\theta_0\|_{L^2}\le K_0,
\end{align}
where $r>3$. There exist a positive time $T_*=T_*(\tau_0, K_0)$ and a unique solution in $\Gs$ on $[0, T_*)$ to the initial value problem associated to \eqref{general abstract active scalar eqn nondiffusive}.

When $r_0=0$, there exist a positive time $T_*=T_*(\tau_0, K_0)$ and a unique solution in $\Gs$ on $[0, T_*)$ to the initial value problem associated to \eqref{general abstract active scalar eqn nondiffusive} for all $s\ge1$.
\end{thm}

\begin{proof}[Proof of Theorem~\ref{local existence Gevrey thm general}]
We only give the proof for $r_0\in(0,1]$, since the proof for the case $r_0=0$ is almost identical. 

Let $\tau=\tau(t)>0$ denotes the radius of convergence and we fix $r>3$, $r_0\in(0,1]$ and $s\in[1,\frac{1}{r_{0}}]$. Throughout this proof, $C$ always denotes a positive generic constant which may depend on $r$, $r_0$, $s$ but is independent of $\tau$, $\bl$, $\bj$ and $\bk$, and for simplicity we write $\Lambda=(-\Delta)^\frac{1}{2}$.

We take $L^2$-inner product of \eqref{general abstract active scalar eqn nondiffusive} with $\Lambda^{2r}e^{2\tau\Lambda^{\frac{1}{s}}}\theta$ and obtain
\begin{align}\label{a priori estimates on theta singular gevrey}
\frac{1}{2}\frac{d}{dt}\|\theta\|^2_{\Gs_\tau}-\dot{\tau}\|\Lambda^{\frac{1}{2s}}\theta\|^2_{\Gs_\tau}=\langle u\cdot\nabla\theta,\Lambda^{2r}e^{2\tau\Lambda^{\frac{1}{s}}}\theta\rangle.
\end{align} 
Write $\mathcal{R}=\langle u\cdot\nabla\theta,\Lambda^{2r}e^{2\tau\Lambda^{\frac{1}{s}}}\theta\rangle$. Since $\nabla\cdot u=0$, we have 
\begin{align*}
\langle u\cdot\nabla \Lambda^r e^{\tau\Lambda^{\frac{1}{s}}}\theta,\Lambda^r e^{\tau\Lambda^{\frac{1}{s}}}\theta\rangle=0
\end{align*}
and hence
\begin{align*}
\mathcal{R}=\langle u\cdot\nabla\theta,\Lambda^{2r}e^{2\tau\Lambda^{\frac{1}{s}}}\theta\rangle-\langle u\cdot\nabla \Lambda^r e^{\tau\Lambda^{\frac{1}{s}}}\theta,\Lambda^r e^{\tau\Lambda^{\frac{1}{s}}}\theta\rangle.
\end{align*}
By Plancherel’s theorem, $\mathcal{R}$ can rewritten as
\begin{align*}
\mathcal{R}&=i(2\pi)^{d}\summ\widehat{u}(\bj)\cdot \bk\widehat{\theta}(\bk)|\bl|^{2r}e^{2\tau|\bl|^\frac{1}{s}}\widehat{\theta}(-\bl)\\
&\qquad-i(2\pi)^{d}\summ\widehat{u}(\bj)\cdot\bk|\bk|^re^{\tau|\bk|^\frac{1}{s}}\widehat{\theta}(\bk)|\bl|^re^{\tau|\bl|^\frac{1}{s}}\widehat{\theta}(-\bl).
\end{align*}
Define $\A$ and $\cB$ by
\begin{align*}
\left\{ \begin{array}{l}
\A:=\{(\bj,\bk,\bl):\bj+\bk=\bl; |\bk|\le|\bj|;\bj,\bk,\bl\in\Z^{d}_*\} \\
\cB:=\{(\bj,\bk,\bl):\bj+\bk=\bl; |\bj|\le|\bk|;\bj,\bk,\bl\in\Z^{d}_*\},
\end{array}\right.
\end{align*}
then $\mathcal{R}$ can be bounded by
\begin{align*}
|\mathcal{R}|\le|\mathcal{R}_1| + |\mathcal{R}_2|,
\end{align*}
where
\begin{align*}
\mathcal{R}_1&=i(2\pi)^{d}\suma\widehat{u}(\bj)\cdot \bk\widehat{\theta}(\bk)|\bl|^{2r}e^{2\tau|\bl|^\frac{1}{s}}\widehat{\theta}(-\bl)\\
&\qquad-i(2\pi)^{d}\suma\widehat{u}(\bj)\cdot\bk|\bk|^re^{\tau|\bk|^\frac{1}{s}}\widehat{\theta}(\bk)|\bl|^re^{\tau|\bl|^\frac{1}{s}}\widehat{\theta}(-\bl)
\end{align*}
and
\begin{align*}
\mathcal{R}_2&=i(2\pi)^{d}\sumb\widehat{u}(\bj)\cdot \bk\widehat{\theta}(\bk)|\bl|^{2r}e^{2\tau|\bl|^\frac{1}{s}}\widehat{\theta}(-\bl)\\
&\qquad-i(2\pi)^{d}\sumb\widehat{u}(\bj)\cdot\bk|\bk|^re^{\tau|\bk|^\frac{1}{s}}\widehat{\theta}(\bk)|\bl|^re^{\tau|\bl|^\frac{1}{s}}\widehat{\theta}(-\bl).
\end{align*}
We estimate $\mathcal{R}_1$ and $\mathcal{R}_2$ as follows.

\noindent{\bf Estimates on $\mathcal{R}_1$:} With the help of inequality \eqref{useful algebraic ineq}, we have
\begin{align*}
|\bl|^\frac{1}{s}\le|\bj|^\frac{1}{s}+|\bk|^\frac{1}{s}
\end{align*}
and 
\begin{align*}
|\bl|^r\le C_{r}(|\bj|^r+|\bk|^r)
\end{align*}
for some $C_{r}$ independent of $\bl$, $\bj$ and $\bk$. Therefore we obtain
\begin{align}\label{estimate on R 1}
\mathcal{R}_1&\le C\suma|\bj|^{r_{0}}|\bk|(|\bj|^r+|\bk|^r)|\widehat{\theta}(\bj)|e^{\tau|\bj|^\frac{1}{s}}|\widehat{\theta}(\bk)|e^{\tau|\bk|^\frac{1}{s}}|\bl|^r|\widehat{\theta}(\bl)|e^{\tau|\bl|^\frac{1}{s}}\notag\\
&\qquad+C\suma|\bj|^{r_{0}}|\widehat{\theta}(\bj)||\bk|^{r+1}e^{\tau|\bk|^\frac{1}{s}}|\widehat{\theta}(\bk)||\bl|^re^{\tau|\bl|^\frac{1}{s}}|\widehat{\theta}(\bl)|\notag\\
&\le C\suma|\bj|^{r+r_{0}}|\bk||\bl|^r|\widehat{\theta}(\bj)|e^{\tau|\bj|^\frac{1}{s}}|\widehat{\theta}(\bk)|e^{\tau|\bk|^\frac{1}{s}}|\widehat{\theta}(\bl)|e^{\tau|\bl|^\frac{1}{s}}\notag\\
&\qquad+C\suma|\bj|^{r_{0}}|\bk|^{r+1}|\bl|^r|\widehat{\theta}(\bj)|e^{\tau|\bj|^\frac{1}{s}}|\widehat{\theta}(\bk)|e^{\tau|\bk|^\frac{1}{s}}|\widehat{\theta}(\bl)|e^{\tau|\bl|^\frac{1}{s}}\notag\\
&\qquad+C\suma|\bj|^{r_{0}}|\widehat{\theta}(\bj)||\bk|^{r+1}e^{\tau|\bk|^\frac{1}{s}}|\widehat{\theta}(\bk)||\bl|^re^{\tau|\bl|^\frac{1}{s}}|\widehat{\theta}(\bl)|.
\end{align}
Notice that since $r_{0}\le\frac{1}{s}$ and $|\bj|\ge1$, we have 
\begin{align*}
|\bj|^{r+r_{0}}\le|\bj|^{r+\frac{1}{s}}. 
\end{align*}
Moreover, using inequality \eqref{useful algebraic ineq}, it implies $|\bj|^\frac{1}{2s}\le|\bk|^\frac{1}{2s}+|\bl|^\frac{1}{2s}$. Hence we obtain
\begin{align*}
&C\suma|\bj|^{r+r_{0}}|\bk||\bl|^r|\widehat{\theta}(\bj)|e^{\tau|\bj|^\frac{1}{s}}|\widehat{\theta}(\bk)|e^{\tau|\bk|^\frac{1}{s}}|\widehat{\theta}(\bl)|e^{\tau|\bl|^\frac{1}{s}}\\
&\le C\suma|\bj|^{r+\frac{1}{2s}}|\bk|(|\bk|^\frac{1}{2s}+|\bl|^\frac{1}{2s})|\bl|^r|\widehat{\theta}(\bj)|e^{\tau|\bj|^\frac{1}{s}}|\widehat{\theta}(\bk)|e^{\tau|\bk|^\frac{1}{s}}|\widehat{\theta}(\bl)|e^{\tau|\bl|^\frac{1}{s}}\\
&\le C\suma|\bj|^{r+\frac{1}{2s}}|\bk|^{1+\frac{1}{2s}}|\bl|^{r+\frac{1}{2s}}|\widehat{\theta}(\bj)|e^{\tau|\bj|^\frac{1}{s}}|\widehat{\theta}(\bk)|e^{\tau|\bk|^\frac{1}{s}}|\widehat{\theta}(\bl)|e^{\tau|\bl|^\frac{1}{s}},
\end{align*}
where the last inequality follows since $|\bk|,|\bl|\ge1$. Similarly, since $|\bk|\le|\bj|$, we also have
\begin{align*}
&C\suma|\bj|^{r_{0}}|\bk|^{r+1}|\bl|^r|\widehat{\theta}(\bj)|e^{\tau|\bj|^\frac{1}{s}}|\widehat{\theta}(\bk)|e^{\tau|\bk|^\frac{1}{s}}|\widehat{\theta}(\bl)|e^{\tau|\bl|^\frac{1}{s}}\\
&\le C\suma|\bj|^{r+\frac{1}{2s}}|\bk|^{1+\frac{1}{2s}}|\bl|^{r+\frac{1}{2s}}|\widehat{\theta}(\bj)|e^{\tau|\bj|^\frac{1}{s}}|\widehat{\theta}(\bk)|e^{\tau|\bk|^\frac{1}{s}}|\widehat{\theta}(\bl)|e^{\tau|\bl|^\frac{1}{s}},
\end{align*}
and
\begin{align*}
&C\suma|\bj|^{r_{0}}|\widehat{\theta}(\bj)||\bk|^{r+1}e^{\tau|\bk|^\frac{1}{s}}|\widehat{\theta}(\bk)||\bl|^re^{\tau|\bl|^\frac{1}{s}}|\widehat{\theta}(\bl)|\\
&\le C\suma|\bj|^{r_{0}+1}|\widehat{\theta}(\bj)||\bk|^{r}e^{\tau|\bk|^\frac{1}{s}}|\widehat{\theta}(\bk)||\bl|^re^{\tau|\bl|^\frac{1}{s}}|\widehat{\theta}(\bl)|.
\end{align*}
Hence we obtain from \eqref{estimate on R 1} that
\begin{align*}
&\mathcal{R}_1\le C\suma|\bj|^{r+\frac{1}{2s}}|\bk|^{1+\frac{1}{2s}}|\bl|^{r+\frac{1}{2s}}|\widehat{\theta}(\bj)|e^{\tau|\bj|^\frac{1}{s}}|\widehat{\theta}(\bk)|e^{\tau|\bk|^\frac{1}{s}}|\widehat{\theta}(\bl)|e^{\tau|\bl|^\frac{1}{s}}\notag\\
&\qquad+C\suma|\bj|^{r_{0}+1}|\widehat{\theta}(\bj)||\bk|^{r}e^{\tau|\bk|^\frac{1}{s}}|\widehat{\theta}(\bk)||\bl|^re^{\tau|\bl|^\frac{1}{s}}|\widehat{\theta}(\bl)|.
\end{align*}
Following the argument given in \cite{FV11b}, we further obtain that 
\begin{align*}
\mathcal{R}_1\le C\|\Lambda^\frac{1}{2s}\theta\|^2_{\Gs_\tau}\sum_{\bj\in\Z_*^d}|\bj|^{1+\frac{1}{2s}}|\widehat{\theta}(\bj)|e^{\tau|\bj|^\frac{1}{s}}+C\|\theta\|^2_{\Gs_\tau}\sum_{\bj\in\Z_*^d}|\bj|^{r_{0}+1}|\widehat{\theta}(\bj)|e^{\tau|\bj|^\frac{1}{s}}.
\end{align*}
By using Cauchy-Schwartz inequality, we have 
\begin{align*}
\sum_{\bj\in\Z_*^d}|\bj|^{1+\frac{1}{2s}}|\widehat{\theta}(\bj)|e^{\tau|\bj|^\frac{1}{s}}\le \Big(\sum_{\bj\in\Z_*^d}|\bj|^{2+\frac{1}{s}-2r}\Big)^\frac{1}{2}\Big(\sum_{\bj\in\Z_*^d}|\bj|^{2r}|\widehat{\theta}(\bj)|e^{2\tau|\bj|^\frac{1}{s}}\Big)^\frac{1}{2}\le C\|\theta\|_{\Gs_\tau},
\end{align*}
where the last inequality holds since $s\ge1$ and $r>3$. Similarly,
\begin{align*}
\sum_{\bj\in\Z_*^d}|\bj|^{r_{0}+1}|\widehat{\theta}(\bj)|e^{\tau|\bj|^\frac{1}{s}}\le C\|\theta\|_{\Gs_\tau},
\end{align*}
and we conclude that
\begin{align}\label{intermediate estimate on R1}
\mathcal{R}_1\le C\|\Lambda^\frac{1}{2s}\theta\|^2_{\Gs_\tau}\|\theta\|_{\Gs_\tau}.
\end{align}
\noindent{\bf Estimates on $\mathcal{R}_2$:} We first bound $\mathcal{R}_2$ by
\begin{align*}
|\mathcal{R}_2|\le \mathcal{T}_1+\mathcal{T}_2
\end{align*}
where $\mathcal{T}_1$ and $\mathcal{T}_2$ are given by
\begin{align*}
\mathcal{T}_1&:=C\sumb||\bl|^r-|\bk|^r||\widehat{u}(\bj)||\widehat{\theta}(\bk)||\bk||\bl|^r|\widehat{\theta}(\bl)|e^{\tau|\bl|^\frac{1}{s}}e^{\tau|\bk|^\frac{1}{s}},\\
\mathcal{T}_2&:=C\sumb|\bl|^r|e^{\tau|\bl|^\frac{1}{s}-\tau|\bk|^\frac{1}{s}}-1||\widehat{u}(\bj)||\widehat{\theta}(\bk)||\bk||\bl|^re^{\tau|l|^\frac{1}{s}}|\widehat{\theta}(\bl)|e^{\tau|\bk|^\frac{1}{s}}.
\end{align*}
To estimate $\mathcal{T}_1$ and $\mathcal{T}_2$, we apply the similar method given in \cite{KV09, MV11, FS21b}. Using mean value theorem, there exists $\xi_{\bk,\bl}\in(0,1)$ such that
\begin{align*}
|\bl|^r-|\bk|^r=((\xi_{\bk,\bl}|l|+(1-\xi_{\bk,\bl})|\bk|)^{r-1}-|\bk|^{r-1})+(|\bl|-|\bk|)|\bk|^{r-1}.
\end{align*}
Since $\bj+\bk=\bl$, we have $|(|\bl|-|\bk|)|\bk|^{r-1}|\le |\bj||\bk|^{r-1}$ as well as
\begin{align*}
|(\xi_{\bk,\bl}|l|+(1-\xi_{\bk,\bl})|\bk|)^{r-1}-|\bk|^{r-1}|\le C|\bj|^2(|\bj|^{r-2}+|\bk|^{r-2}).
\end{align*}
Together with the \eqref{singular operator condition}, we have
\begin{align*}
\mathcal{T}_1&\le C\sumb|\bj|^{\alpha+3}(|\bj|^{r-2}+|\bk|^{r-2})|\widehat{\theta}(\bj)||\widehat{\theta}(\bk)||\bk||\bl|^r|\widehat{\theta}(\bl)|e^{\tau|\bl|^\frac{1}{s}}e^{\tau|\bk|^\frac{1}{s}}\\
&\qquad+C\sumb|\bj|^{r_{0}+1}|\bk|^{r-1}|\widehat{\theta}(\bj)||\widehat{\theta}(\bk)||\bk||\bl|^r|\widehat{\theta}(\bl)|e^{\tau|\bl|^\frac{1}{s}}e^{\tau|\bk|^\frac{1}{s}}\\
&\le C\|\theta\|^2_{\Gs_\tau}\|\Lambda^r\theta\|_{L^2},
\end{align*}
where the last inequality holds for $\alpha\le0$ and $r>3$. To estimate $\mathcal{T}_2$, using \eqref{singular operator condition} and the inequalities
\begin{align*}
|e^{\tau|\bl|^\frac{1}{s}-\tau|\bk|^\frac{1}{s}}-1|\le |\tau|\bl|^\frac{1}{s}-\tau|\bk|^\frac{1}{s}|e^{|\tau|\bl|^\frac{1}{s}-\tau|\bk|^\frac{1}{s}|}
\end{align*}
and
\begin{align*}
|\tau|\bl|^\frac{1}{s}-\tau|\bk|^\frac{1}{s}|\le\tau|\bj|^\frac{1}{s},
\end{align*}
we obtain
\begin{align*}
\mathcal{T}_2
&\le C\tau\sumb|\bl|^{2r}|\bk||\bj|^{r_{0}}||\bl|^\frac{1}{s}-|\bk|^\frac{1}{s}|e^{\tau|\bj|^\frac{1}{s}}|\widehat{\theta}(\bj)||\widehat{\theta}(\bk)||\widehat{\theta}(\bl)|e^{\tau|\bl|^\frac{1}{s}}e^{\tau|\bk|^\frac{1}{s}},
\end{align*}
Since $|\bj|\le|\bk|$ and $\bl=\bk+\bj$, we have $|\bl|\le 2|\bk|$. And for $s\ge1$, using the inequality 
\begin{align*}
||\bl|^\frac{1}{s}-|\bk|^\frac{1}{s}|\le\frac{C|\bj|}{|\bl|^{1-\frac{1}{s}}+|\bk|^{1-\frac{1}{s}}},
\end{align*}
we obtain
\begin{align*}
&|\bl|^{2r}|\bk||\bj|^{r_{0}}||\bl|^\frac{1}{s}-|\bk|^\frac{1}{s}|e^{\tau|\bj|^\frac{1}{s}}|\widehat{\theta}(\bj)||\widehat{\theta}(\bk)||\widehat{\theta}(\bl)|e^{\tau|\bl|^\frac{1}{s}}e^{\tau|\bk|^\frac{1}{s}}\\
&\le C|\bk|^{r+1-\frac{1}{2s}} e^{\tau|\bk|^\frac{1}{s}} \Big(\frac{|\bj|}{|\bl|^{1-\frac{1}{s}}+|\bk|^{1-\frac{1}{s}}}\Big)|\bj|^{r_{0}}|\widehat{\theta}(\bj)||\widehat{\theta}(\bk)||\widehat{\theta}(\bl)||\bl|^{r+\frac{1}{2s}}e^{\tau|\bl|^\frac{1}{s}}e^{\tau|\bk|^\frac{1}{s}}\\
&\le C\Big(|\bj|^{r_{0}+1}e^{\tau|\bj|^\frac{1}{s}}|\widehat{\theta}(\bj)|\Big)\Big(|\bk|^{r+\frac{1}{2s}}e^{\tau|\bk|^\frac{1}{s}}|\widehat{\theta}(\bk)|\Big)\Big(|\bl|^{r+\frac{1}{2s}}e^{\tau|\bl|^\frac{1}{s}}|\widehat{\theta}(\bl)|\Big).
\end{align*}
Hence we have
\begin{align*}
\mathcal{T}_2\le C\tau\|\Lambda^{\frac{1}{2s}}\theta\|^2_{\Gs_\tau}\|\theta\|_{\Gs_\tau},
\end{align*}
and we conclude that
\begin{align}\label{intermediate estimate on R2}
\mathcal{R}_2\le C\|\theta\|^2_{\Gs_\tau}\|\Lambda^r\theta\|_{L^2}+C\tau\|\Lambda^{\frac{1}{2s}}\theta\|^2_{\Gs_\tau}\|\theta\|_{\Gs_\tau}.
\end{align}
Combining the estimates \eqref{a priori estimates on theta singular gevrey}, \eqref{intermediate estimate on R1} and \eqref{intermediate estimate on R2}, for $\tau<1$, we obtain the {\it a priori} bound
\begin{align}\label{a priori bound on theta singular gevrey}
\frac{1}{2}\frac{d}{dt}\|\theta(\cdot,t)\|^2_{\tau(t)}\le (\dot{\tau}(t)+C\|\theta(\cdot,t)\|_{\tau(t)})\|\Lambda^\frac{1}{2}\theta(\cdot,t)\|^2_{\tau(t)}.
\end{align}
Let $\tau=\tau(t)$ be deceasing and satisfy 
\begin{align*}
\dot{\tau}+2CK_0=0,
\end{align*}
with initial condition $\tau(0)=\tau_0$, then we have $\dot{\tau}+C\|\theta(\cdot,t)\|_{\Gs_\tau}<0$, and from \eqref{a priori bound on theta singular gevrey} that 
\begin{align}\label{bound on solution nu=0}
\|\theta(\cdot,t)\|_{\Gs_\tau}\le\|\theta(\cdot,0)\|_{\Gs_\tau}\le K_0
\end{align}
as long as $\tau>0$. Hence by the standard Picard iteration argument, it implies the existence of a solution $\theta\in\Gs$ on $[0,T_*)$, where the maximal time of existence of solution in $\Gs$ is given by $T_*=\frac{\tau_0}{2CK_0}$.
\end{proof}

\section{Illposedness in Gevrey-class of active scalar equations}\label{illposed in Gevrey class general}

In this section, we address the equation \eqref{general abstract active scalar eqn nondiffusive} when $r_0\in[1,2]$. For the Fourier symbol $\widehat{\mathbf{T}}$, we further assume the following conditions holds:
\begin{itemize}
\item[(C1)] $\widehat{\mathbf{T}}(\0', a) = 0$ for a given positive integer $a$, that is, $\sin(ax_d)$ is a steady state solution of \eqref{general abstract active scalar eqn nondiffusive} with corresponding velocity $u = 0$,
\item[(C2)] $\widehat{\mathbf{T}}_{d}(\bk)$ is a real positive rational function and is even in $\bk$.
\end{itemize}
Moreover, there exists a sequence $\{\bbj\}_{j\in\N}$ with $\bbj=(b_1^{(j)},\cdots,b_{d-1}^{(j)})\in\mathbb{Z}^{d-1}$ and $\dis\lim_{j\to\infty}|\bbj|=\infty$ such that
\begin{itemize}
\item[(C3)] $\widehat{\mathbf{T}}_{d}(\bbj,na)\to0$ as $n\to\infty$ for any fixed $j\in\N$,
\item[(C4)] $\widehat{\mathbf{T}}_{d}(\bbj,(n+1)a)<\widehat{\mathbf{T}}_{d}(\bbj,na)$ for all $n\in\mathbb{N}$ and any fixed $j\in\N$,
\item[(C5)] $\widehat{\mathbf{T}}_{d}(\bbj,na)\le \tilde C_1|\bbj|^{\beta_{1}}n^{\beta_{2}}$ for all $j$, $n\in\N$,
\item[(C6)] $\widehat{\mathbf{T}}_{d}(\bbj,a)\widehat{\mathbf{T}}_{d}(\bbj,2a)\ge\tilde C_2|\bbj|^{2\beta_{3}}$ for all $j\in\N$,
\end{itemize}
where $\beta_{1}$, $\beta_{2}$, $\beta_{3}\in\R$ satisfy
\begin{align}
\label{conditions on betas} 
\left\{ \begin{array}{l}
0<\beta_{3}\le\beta_{1}, \\
\beta_{1}+\beta_{2}\le r_{0}, \\
-2\le\beta_{2}< 0,
\end{array}\right.
\end{align}
and $\tilde C_1$, $\tilde C_2>0$ are constants which are independent of $j$ and $n$. 

\begin{rem}\label{discussion on the conditions C1 to C6}
Conditions (C1)--(C6) are reminiscent of those given in \cite[Section D]{FGSV12}. Examples of physical models satisfying conditions (C1)--(C6) include the magneto-geostrophic (MG) equation and the singular incompressible porous media (SIPM) equation; refer to Section~\ref{Applications to physical models section} for more detailed discussions. We point out that conditions (C3)--(C6) are required to hold {\it only for} some sequence $\{\bbj\}_{j\in\N}$ with $\dis\lim_{j\to\infty}|\bbj|=\infty$, which are somewhat more general than those given in \cite[Section D]{FGSV12}. 
\end{rem}

Under the conditions (C1)--(C6), we claim that solution to \eqref{general abstract active scalar eqn nondiffusive} is not well-posed in a class of Gevrey spaces. 

\begin{thm}\label{local ill-posed Gevrey thm general}
Let $r_{0}\in[1,2]$ and assume that \eqref{singular operator condition} holds. Fix $q>r_{0}+\frac{d}{4}$ and $s>\frac{\beta_{3}-\beta_{1}}{\beta_{3}\beta_{2}}$. Under the conditions {\rm(C1)--(C6)}, the equation \eqref{general abstract active scalar eqn nondiffusive} is locally Lipschitz $(H^q , \Gs_{\tau})$ ill-posed for $\tau > 0$.
\end{thm}

\begin{rem}
Theorem~\ref{local ill-posed Gevrey thm general} is inspired by the results given in \cite{KVW16} where a class of active scalar equations with odd singular kernels were studied. In fact, due to the special cancelation property of odd constitutive law, these equations are prone to be well-posed. Therefore, the kernels need to be more singular (of order 2). In this work, we are interested in even kernels and hence we are able to get the ill-posedness result with constitutive law of order between 1 to 2, which is consistent with the scenarios previously considered.
\end{rem}

In order to prove Theorem~\ref{local ill-posed Gevrey thm general}, we first linearize the equations \eqref{general abstract active scalar eqn nondiffusive} about a steady state $\st$, and then show that there exist unstable eigenvalues with arbitrarily large real part. Once these eigenvalues are exhibited, we apply a generic perturbative argument to show that this severe linear ill-posedness implies the Lipschitz ill-posedness for the nonlinear problem. 

To begin with, we fix $a\in\N$ and define
\begin{align}\label{def of steady state general class}
\st:=\sin(ax_{d}),
\end{align}
then by condition (C1), $\st$ is a steady state of \eqref{general abstract active scalar eqn nondiffusive}. Upon linearizing \eqref{general abstract active scalar eqn nondiffusive}, we consider the linear evolution of the perturbation $\bt=\theta-\st$ which is given by
\begin{align}\label{linearized active scalar general class}
\dt\bt + (\widehat{\mathbf{T}}_{d}\widehat{\bt})^{\vee}(\bx)a\cos(ax_{d}) = 0.
\end{align}
Let $\{\bbj\}_{j\in\N}$ be the sequence with $\bbj=(b_1^{(j)},\cdots,b_{d-1}^{(j)})\in\mathbb{Z}^{d-1}$ and $\dis\lim_{j\to\infty}|\bbj|=\infty$ such that conditions (C3)--(C6) hold. For each $j$, we construct eigenvalue-eigenfunction pair $(\sigma,\phi)=(\sigma^{(\bbj)},\phi^{(\bbj)})$ for the following equation
\begin{align}\label{def of linear operator L eigenfunction general class}
\mathbf{L}\phi=\sigma\phi,
\end{align}
where the linear operator $\mathbf{L}$ is given by 
\begin{align}\label{def of linear operator L}
\mathbf{L}\phi=-(\widehat{\mathbf{T}}_{d}\widehat{\phi})^{\vee}(\bx)a\cos(ax_{d})
\end{align}
and $\phi$ is given by
\begin{align}\label{form of eigenfunction general class}
\phi(\bx)=\prod_{i=1}^{d-1}\sin(b_{i}^{(j)}x_{i})\sum_{p\ge1}c_{p}\sin(pax_{d})
\end{align}
and $c_p$ will be determined later. We claim that the following theorem holds for all $r_{0}\in[1,2]$:

\begin{thm}\label{linear illposedness for L nondiffusive thm general}
Let $r_{0}\in[1,2]$. The linearized operator $\mathbf{L}$ defined in \eqref{def of linear operator L} has a sequence of entire real-analytic eigenfunctions $\{\phi^{(\bbj)}\}_{j\in\N}$ with corresponding eigenvalues $\{\sigma^{(\bbj)}\}_{j\in\N}$, such that
\begin{align}\label{refined lower bound on eigenvalues general}
\sigma^{(\bbj)}> \frac{2\sqrt{\tilde C_2}}{a}|\bbj|^{\beta_{3}},
\end{align}
for all $j\in\N$. Moreover, we can normalize $\phi^{(\bbj)}$ so that given $s\ge1$ and $\tau>0$, we have
\begin{align}\label{lower bound on L2 norm on eigenfunction in terms of s general}
\|\phi^{(\bbj)}\|_{G^{s}_{\tau}}=1,\qquad \|\phi^{(\bbj)}\|_{L^2}\ge C_{s,\tau}^{-1}\exp\Big(-C_{s,\tau} |\bbj|^\frac{\beta_{3}-\beta_{1}}{s\beta_{2}}\Big),
\end{align}
for all $j\in\N$, where $C_{s,\tau}\ge1$ is a constant which depends on $s$ and $\tau$ but is independent of $j$.
\end{thm} 

\begin{proof}
Fix $j\in\N$. From \eqref{def of linear operator L eigenfunction general class}, \eqref{def of linear operator L} and \eqref{form of eigenfunction general class}, we obtain the recursion relations
\begin{align}
\sigma c_1+\frac{c_2}{\alpha_2}&=0,\label{recurrence relation 1}\\
\sigma c_p+\frac{c_{p+1}}{\alpha_{p+1}}+\frac{c_{p-1}}{\alpha_{p-1}}&=0,\qquad p\ge2,\label{recurrence relation 2}
\end{align}
with $\alpha_{p}$ being explicitly given by
\begin{align}\label{def of alpha p}
\alpha_{p}:=\frac{2/a}{\widehat{\mathbf{T}}_{d}(\bbj,pa)},\qquad p\ge1.
\end{align}
To prove the existence of $\sigma$, we define $\eta_2=-\sigma\alpha_1$ and
\begin{align}\label{def of eta p}
\eta_p=\frac{c_p\alpha_{p-1}}{c_{p-1}\alpha_p},\qquad p\ge3
\end{align}
then \eqref{recurrence relation 1}-\eqref{recurrence relation 2} can be rewritten as
\begin{align}
\sigma\alpha_1+\eta_2&=0,\label{recurrence relation in eta 1}\\
\sigma \alpha_p+\eta_{p+1}+\frac{1}{\eta_p}&=0,\qquad p\ge2.\label{recurrence relation in eta 2}
\end{align}
Using \eqref{recurrence relation in eta 1}-\eqref{recurrence relation in eta 2}, we have
\begin{align}\label{continued fraction for sigma1alpha1}
-\eta_{2}=\sigma\alpha_1=\frac{1}{\sigma\alpha_2+\eta_3}=\frac{1}{\sigma\alpha_2-\frac{1}{\sigma\alpha_3+\eta_4}}=\frac{1}{\sigma\alpha_2-\frac{1}{\sigma\alpha_3-\frac{1}{\sigma\alpha_4-\dots}}},
\end{align}
hence we can write $\eta_2$ as an infinite continued fraction.

For real values of $\sigma$, we define the infinite continued fraction $F_p(\sigma)$ by
\begin{align}\label{continued fraction for Fp(s)}
F_p(\sigma)=\frac{1}{\sigma\alpha_p-\frac{1}{\sigma_{p+1}\alpha_{p+1}-\frac{1}{\sigma_{p+2}\alpha_{p+2}-\dots}}}
\end{align}
and the function $G_p(\sigma)$ by
\begin{align}\label{def of Gp(s)}
G_p(\sigma)=\frac{2}{\sigma+\sqrt{\sigma^2\alpha_p^2-4}}
\end{align}
for all $p\ge2$ and all $\sigma$ such that $\sigma\alpha_2>2$. We note that $F_p(\sigma)$ is well-defined and smooth except for a set of points on the real axis with $\sigma\alpha_2<2$, and for the rest of the proof we will always assume that $\sigma$ is real with $\sigma\alpha_2>2$. Notice that $G_p$ also satisfies
\begin{align}
G_p(\sigma)=\frac{1}{\sigma\alpha_p-G_p(\sigma)}=\frac{1}{\sigma\alpha_p-\frac{1}{\sigma\alpha_p-\frac{1}{\sigma\alpha_p-\dots}}}.
\end{align} 
Using condition (C4) and  following an inductive argument, one can show that $\{G_p(\sigma)\}_{p\ge2}$ is a monotonic decreasing sequence with
\begin{align*}
G_2(\sigma)>G_3(\sigma)>G_4(\sigma)>\dots\ge0.
\end{align*}
Since by condition (C3) $\alpha_p\to\infty$ as $p\to\infty$, it implies $G_p(\sigma)\to0$ as $p\to\infty$ for every fixed $\sigma$. Moreover, for all $p\ge2$, 
\begin{align}\label{Gp(s)>Fp(s)}
G_p(\sigma)>\frac{1}{\sigma\alpha_p-G_{p+1}(\sigma)}&>\frac{1}{\sigma\alpha_p-\frac{1}{\sigma\alpha_{p+1}-G_{p+2}(\sigma)}}\notag\\
&>\dots>\frac{1}{\sigma\alpha_p-\frac{1}{\sigma_{p+1}\alpha_{p+1}-\frac{1}{\sigma_{p+2}\alpha_{p+2}-\dots}}}=F_p(\sigma),
\end{align}
which shows that $F_p(\sigma)\to0$ as $p\to\infty$. Hence $F_p(\sigma)\ge0$ and in particular we have $F_2(\sigma)\ge0$ with 
\begin{align}\label{lower bound for G2}
\frac{1}{\sigma\alpha_2}<F_2(\sigma)<G_2(\sigma).
\end{align}
Define $H(\sigma)=F_2(\sigma)-\sigma\alpha_1$. Then $H(\sigma)$ is continuous for all $\sigma\alpha_2>2$ with $H(\sigma)\to-\infty$ as $\sigma\to\infty$. Also, if $\dis\sigma<\frac{1}{\sqrt{\alpha_1\alpha_2}}$, then we have
\begin{align*}
H(\sigma)>\frac{1}{\sigma\alpha_2}-\sigma\alpha_1=\frac{1-\sigma^2\alpha_1\alpha_2}{\sigma\alpha_2}>0,
\end{align*}
hence there exists $\sigma_0>0$ such that $H(\sigma_0)>0$. Therefore, by intermediate value theorem, there exists $\sigma=\sigma^{(\bbj)}\in(\sigma_0,\infty)$ such that $H(\sigma)=0$, which implies
\begin{align}\label{eqn for F2(sigma)}
F_2(\sigma)=\sigma\alpha_1.
\end{align}
Furthermore, using \eqref{lower bound for G2}, we obtain an upper and a lower bound on $\sigma$:
\begin{align}\label{lower and upper bound on sigma}
\frac{1}{\sqrt{\alpha_1\alpha_2}}<\sigma<\frac{1}{\sqrt{\alpha_1\alpha_2-\alpha_1^2}}.
\end{align}
Together with condition (C6), we have
\begin{align}\label{bound on 1/sigma}
\frac{1}{\sigma}<\frac{2/a}{\sqrt{\tilde C_2}}|\bbj|^{-\beta_{3}},
\end{align}
and hence \eqref{bound on 1/sigma} implies \eqref{refined lower bound on eigenvalues general}.

Next, in order to prove \eqref{lower bound on L2 norm on eigenfunction in terms of s general}, we first construct the sequence $\{c_p\}_{p\ge1}$. We let $c_1=\alpha_1$ and 
\begin{align}\label{def of cp}
c_p=\alpha_p\eta_p\eta_{p-1}\cdots\eta_2,\qquad p\ge2.
\end{align}
Then $\{c_p\}_{p\ge1}$ satisfies \eqref{recurrence relation 1}-\eqref{recurrence relation 2} with
\begin{align*}
\eta_p=\frac{-1}{\sigma\alpha_p+\eta_{p+1}}=\frac{-1}{\sigma\alpha_p-\frac{1}{\sigma\alpha_{p+1}-\frac{1}{\sigma\alpha_{p+2}-\dots}}}=-F_p(\sigma).
\end{align*}
Moreover, similar to \eqref{Gp(s)>Fp(s)}-\eqref{lower bound for G2}, one can show that 
\begin{align*}
G_p(\sigma)>F_p(\sigma)>\frac{1}{\sigma\alpha_p}
\end{align*}
for all $\sigma$ such that $\sigma_2\alpha_2>2$, which further gives
\begin{align}\label{bounds on eta p}
\frac{-2}{\sigma\alpha_p+\sqrt{\sigma^2\alpha_{p}^2-4}}<\eta_p<\frac{-1}{\sigma\alpha_p}.
\end{align}
Hence \eqref{bounds on eta p} implies
\begin{align}\label{bounds on |eta p|}
|\eta_p|<\frac{2}{\sigma\alpha_p}.
\end{align}
We now estimate $c_p$ as follows. Using \eqref{def of cp} and \eqref{bounds on |eta p|}, we can bound $|c_p|$ by
\begin{align}\label{bound on cp step 1}
|c_p|<\frac{c_12^{p-1}}{\sigma^{p-1}\alpha_{p-1}\cdots\alpha_1}.
\end{align}
Using \eqref{def of alpha p}, \eqref{bound on 1/sigma} and condition (C5),
\begin{align}\label{bound on 1/product of alpha}
\frac{1}{\alpha_{p-1}\cdots\alpha_{1}}<\Big(\frac{a}{2}\Big)^{p-1}(\tilde C_2)^{p-1}|\bbj|^{(p-1)\beta_{1}}\Big((p-1)!\Big)^{\beta_{2}}.
\end{align}
Therefore, we can combine \eqref{bound on 1/sigma}, \eqref{bound on cp step 1} and \eqref{bound on 1/product of alpha} to obtain
\begin{align}\label{bound on cp step 2}
|c_p|<c_1p^{-\beta_{2}}\Big(\frac{2\tilde C_1}{\sqrt{\tilde C_2}}\Big)^{p-1}|\bbj|^{\beta_{1}-\beta_{3}}(p!)^{\beta_{2}}.
\end{align}
By the assumption $-2\le\beta_{2}<0$ and the fact that $\dis p!\ge\Big(\frac{p}{C}\Big)^p$ for some sufficiently large constant $C$, we further obtain from \eqref{bound on cp step 2} that
\begin{align}\label{bound on cp step 3}
|c_p|<p^2\exp\Big(p\log\Big(\bar C|\bbj|^{\beta_{1}-\beta_{3}}p^{\beta_{2}}\Big)\Big),
\end{align}
where $\bar C$ is a positive constant which depends on $a$, $c_1$, $C$, $\tilde C_1$ and $\tilde C_2$ only. Define $p_{*}=p_{*}(\bar C,\tau,\bbj,\beta_{1},\beta_{2},\beta_{3})$ by
\begin{align}\label{def of p*}
p_{*}=\bar C \exp(4\tau)|\bbj|^\frac{\beta_{3}-\beta_{1}}{\beta_{2}},
\end{align}
then we have from \eqref{bound on cp step 3} that
\begin{align}\label{bound on cp step 4}
|c_p|<(p^2+|\bbj|^2)\exp\Big(-2\tau(p^2+|\bbj|^2)^\frac{1}{2}\Big),\qquad p\ge p_*.
\end{align}
Hence using \eqref{bound on cp step 4} and the definition of $\|\cdot\|_{G^1_\tau}$, we have
\begin{align*}
\|\phi\|^2_{G^1_\tau}&=\sum_{p\ge1}c_p^2(p^2+|\bbj|^2)^r\exp\Big(2\tau(p^2+|\bbj|^2)^\frac{1}{2}\Big)\\
&\le\sum_{p\le p_*}(p^2+|\bbj|^2)^{r+2}\exp\Big(2\tau(p^2+|\bbj|^2)^\frac{1}{2}+2p\log\Big(\frac{\bar C|\bbj|^{\beta_{3}-\beta_{1}}}{p^{\beta_{2}}}\Big)\Big)\\
&\qquad+\sum_{p>p_*}(p^2+|\bbj|^2)^{r+2}\exp\Big(-2\tau(p^2+|\bbj|^2)^\frac{1}{2}\Big)<\infty,
\end{align*}
which shows that $\phi$ is entire real-analytic. Moreover, for $s\ge1$ and $\tau>0$, there exists s sufficiently large constant $C_{s,\tau}$ such that
\begin{align}\label{bound on Gevrey norm final general class}
\|\phi\|^2_{G^s_{\tau}}&=\sum_{p\ge1}c_p^2(p^2+|\bbj|^2)^r\exp\Big(2\tau(p^2+|\bbj|^2)^\frac{1}{2s}\Big)\notag\\
&\le p_*^{r+2}\exp\Big(4\tau p_*^\frac{1}{s}\Big)\sum_{p\le p_*}c_p^2\notag\\
&\qquad+\sum_{p>p_*}(p^2+|\bbj|^2)^{r+2}\exp\Big(2\tau(p^2+|\bbj|^2)^\frac{1}{2s}-4\tau(p^2+|\bbj|^2)^\frac{1}{2}\Big)\notag\\
&\le 2C_{s,\tau}\exp\Big(C_{s,\tau}|\bbj|^\frac{\beta_{3}-\beta_{1}}{s\beta_{2}}\|\phi\|^2_{L^2}\Big).
\end{align}
By renormalizing $\phi$, we conclude that \eqref{lower bound on L2 norm on eigenfunction in terms of s general} holds for $\phi$.
\end{proof}

Once the linear ill-posedness is obtained, Theorem~\ref{local ill-posed Gevrey thm general} can now be proved by applying the perturbative argument given in \cite{FGSV12} and we summarise it as follows.

\begin{proof}[Proof of Theorem~\ref{local ill-posed Gevrey thm general}]
We define the nonlinear operator $\mathbf{N}$ by 
\begin{align}\label{def of nonlinear part}
\mathbf{N}\theta=-\mathbf{T}\theta\cdot\nabla\theta,
\end{align}
then by the Sobolev inequality \eqref{Sobolev inequality}, for $q>\frac{d}{4}+r_0$ with $d\ge2$ and $r_0\in[1,2]$, we have
\begin{align*}
\|\mathbf{L}\theta\|_{L^2}\le C\|\theta\|_{H^{q}}
\end{align*}
and
\begin{align*}
\|\mathbf{N}\theta\|_{L^2}\le \|\mathbf{T}\theta\|_{L^4}\|\nabla\theta\|_{L^4}\le C\|\theta\|_{H^{q}}^2,
\end{align*}
for some constant $C>0$. For simplicity, we take $X=H^{q}$ and $Y=G^{s}_{\tau}$.

The rest of the proof will be argued by contradiction: assume that the Cauchy problem \eqref{general abstract active scalar eqn nondiffusive} is Lipschitz locally well-posed in $(H^q , \Gs_{\tau})$ and we will derive a contradiction from it. Fix the steady state $\st$ as given by \eqref{def of steady state general class} and denote a smooth function $\psi_0$ with $\|\psi_0\|_{G^{s}_{\tau}}=1$ which will be chosen later. Define $\theta_0^{\varepsilon}=\theta_0^{(1,\varepsilon)}$ and $\theta_0^{(2)}$ by
\begin{align*}
\theta_0^{\varepsilon}&=\st(x_d)+\varepsilon\psi_0(\bx),\qquad \varepsilon\in(0,\|\st\|_{Y})\\
\theta_0^{(2)}&=\st(x_d),
\end{align*}
and take $\theta^{(2)}(\bx,t)=\st(x_d)$ for all $t>0$. By the Definition~\ref{def of locally Lipschitz wellposedness}, for every $\varepsilon\in(0,\|\st\|_{Y})$, there exists a positive time $T=T(\|\st\|_{Y},\|\theta_0^{\varepsilon}\|_{Y})$ and a positive Lipschitz constant $K=K(\|\st\|_{Y},\|\theta_0^{\varepsilon}\|_{Y})$ such that by the choice of $\psi_0$ and \eqref{condition for solution being Lipschitz wellposed}, we have
\begin{align}\label{bound on the difference by K e}
\sup_{t\in[0,T]}\|\theta^{\varepsilon}(\cdot,t)-\theta^{(2)}(\cdot,t)\|_{X}=\|\theta^{\varepsilon}(\cdot,t)-\st(\cdot)\|_{X}\le K\varepsilon.
\end{align}
And since $\|\theta_0^{\varepsilon}\|_{Y}\le 2\|\st\|_{Y}$, by the continuity in $T$ and $K$, we can choose both $T$ and $K$ independent of $\varepsilon$ such that \eqref{bound on the difference by K e} holds on $[0,T]$. Define
\begin{align*}
\psi^{\varepsilon}(\bx,t)=\frac{\theta^{\varepsilon}-\st}{\varepsilon},
\end{align*}
then applying the argument given in \cite{FGSV12}, there exists a function $\psi\in L^\infty(0,T;X)$ such that $\psi^{\varepsilon}\to\psi$ strongly in $L^2$ and the following system \eqref{linear problem for psi 1}-\eqref{linear problem for psi 2} holds uniquely in $L^\infty(0,T;L^2)$:
\begin{align}
\dt\psi&=\mathbf{L}\psi\label{linear problem for psi 1}\\
\psi(\bx,0)&=\psi_0(\bx).\label{linear problem for psi 2}
\end{align}
Moreover, due to the bound \eqref{bound on the difference by K e}, the function $\psi$ satisfies
\begin{align}\label{bound on psi by K}
\|\psi(\cdot,t)\|_{L^2}\le K,\qquad t\in[0,T].
\end{align}
We now take $\psi_0=\phi^{\bbj}$ where $\phi^{\bbj}$ is an eigenfunction of $\mathbf{L}$ as given in Theorem~\ref{linear illposedness for L nondiffusive thm general} with eigenvalue $\sigma^{\bbj}$, then $\dis\psi(\bx,t)=\exp(\sigma^{\bbj} t)\phi^{\bbj}(\bx)$ with
\begin{align*}
\|\psi(\cdot,\frac{T_0}{2})\|_{L^2}=\exp\Big(\frac{\sigma^{\bbj} T_0}{2}\Big)\|\phi^{\bbj}\|_{L^2}.
\end{align*}
Using the lower bound \eqref{refined lower bound on eigenvalues general} on $\sigma^{\bbj}$, we further obtain
\begin{align}\label{bound on psi general}
\|\psi(\cdot,\frac{T_0}{2})\|_{L^2}>\exp\Big(\frac{a\sqrt{\tilde C_2}}{4}|\bbj|^{\beta_{3}}T_0\Big)\|\phi^{\bbj}\|_{L^2}.
\end{align}
Using \eqref{lower bound on L2 norm on eigenfunction in terms of s general} and \eqref{bound on psi general}, for $s>\frac{\beta_{3}-\beta_{1}}{\beta_{3}\beta_{2}}$, we can choose $j$ large enough so that
\begin{align}\label{bound for contradiction}
\frac{\exp\Big(\frac{a\sqrt{\tilde C_2}}{4}|\bbj|^{\beta_{3}}T_0\Big)}{C_{s,\tau}\exp\Big(C_{s,\tau} |\bbj|^\frac{\beta_{3}-\beta_{1}}{s\beta_{2}}\Big)}\ge 2K,
\end{align}
where $C_{s,\tau}$ is the constant from \eqref{lower bound on L2 norm on eigenfunction in terms of s general}. Using \eqref{bound on psi general} and \eqref{bound for contradiction}, we then have
\begin{align*}
\|\psi(\cdot,\frac{T_0}{2})\|_{L^2}=\exp\Big(\frac{\sigma^{(\bbj)}T_0}{2}\Big)\|\phi^{(\bbj)}\|_{L^2}\ge 2K,
\end{align*}
which contradicts \eqref{bound on psi by K}. This finishes the proof of Theorem~\ref{local ill-posed Gevrey thm general}.
\end{proof}

\begin{rem}
One can also consider the following system with fractional dissipation:
 \begin{align}
\label{general abstract active scalar eqn fractional diffusive} 
\left\{ \begin{array}{l}
\partial_t\theta+u\cdot\nabla\theta=S-\kappa(-\Delta)^\gamma\theta, \\
\nabla\cdot u=0,\,\,\,u=\mathbf{T}\theta,\,\,\,\theta(\bx,0)=\theta_0(\bx),
\end{array}\right.
\end{align}
where $\kappa>0$, $\gamma\in(0,1)$ are constants and $S=S(\bx,t)$ is a given $C^{\infty}$-smooth source
term. Following the similar method as given in the proof of Theorem~\ref{local ill-posed Gevrey thm general}, we can show that the system \eqref{general abstract active scalar eqn fractional diffusive} is illposed in Gevrey class for $s>\frac{\beta_{3}-\beta_{1}}{\beta_{3}\beta_{2}}$ and $\gamma\in(0,\min\{1,\frac{\beta_3}{2}\})$. More precisely, we have
\begin{cor}\label{coro for fractional case}
Let $r_{0}\in[1,2]$ and assume that \eqref{singular operator condition} holds. Fix $q>r_{0}+\frac{d}{4}$, $s>\frac{\beta_{3}-\beta_{1}}{\beta_{3}\beta_{2}}$, $\kappa>0$ and $\gamma\in(0,\min\{1,\frac{\beta_3}{2}\})$. Under the conditions {\rm(C1)--(C6)}, the equation \eqref{general abstract active scalar eqn fractional diffusive} is locally Lipschitz $(H^q , \Gs_{\tau})$ ill-posed for $\tau > 0$.
\end{cor}
We point out that, the result claimed in Corollary~\ref{coro for fractional case} is consistent with previous results for the fractionally dissipative MG equation \cite{FRV12}.
\end{rem}

\section{Applications to physical models}\label{Applications to physical models section}

\subsection{The non-diffusive MG equations}\label{non-diffusive MG equations subsection}

We consider the following class of active scalar equations in $\mathbb{T}^3\times(0,\infty)$:
\begin{align}
\label{abstract active scalar eqn MG}  
\left\{ \begin{array}{l}
\partial_t\theta+u\cdot\nabla\theta=0 \\
u=(-\Delta)^\frac{\alpha}{2}\mathbf{M}\theta,\,\,\,\theta(\bx,0)=\theta_0(\bx),
\end{array}\right.
\end{align}
where $\alpha\in[-1,1]$. Here $\mathbf{M}=(\mathbf{M}_1,\mathbf{M}_2,\mathbf{M}_3)$ is a Fourier multiplier operator with symbols $\mathbf{M}_1$, $\mathbf{M}_2$, $\mathbf{M}_3$ given explicitly by \eqref{MG Fourier symbol_1}-\eqref{MG Fourier symbol_3} in Section~\ref{introduction}. It can be shown that $u$ is divergence-free and $(-\Delta)^\frac{\alpha}{2} \mathbf{M}$ is an even, singular operator of order $\alpha+1$ with 
\begin{align*}
|\widehat{(-\Delta)^\frac{\alpha}{2} \mathbf{M}}(\bk)|\le C(\beta,\eta,\Omega)|\bk|^{\alpha+1},
\end{align*}
where $\beta$, $\eta$, $\Omega$ are the constants appeared in \eqref{MG Fourier symbol_1}-\eqref{MG Fourier symbol_3} and $C(\beta,\eta,\Omega)$ is a fixed positive constant which is independent of $\alpha$ and $\bk$. Notice that for $\alpha=0$, the system \eqref{abstract active scalar eqn MG} reduces to the non-diffusive MG equation given by \eqref{abstract active scalar eqn MG intro}-\eqref{MG Fourier symbol_3}. 

For the case when $\alpha\in[-1,0]$, as a direct consequence of Theorem~\ref{local existence Gevrey thm general}, we immediately obtain the following result:

\begin{thm}\label{local existence Gevrey MG thm}
For $\alpha\in(-1,0]$, we fix $s\in[1,\frac{1}{\alpha+1}]$ and $K_0>0$. Let $\theta_0\in \Gs$ with radius of convergence $\tau_0>0$ and
\begin{align*}
\|(-\Delta)^\frac{r}{2}e^{\tau_0(-\Delta)^\frac{1}{2s}}\theta_0\|_{L^2}\le K_0,
\end{align*}
where $r>3$. There exist $T_*=T_*(\tau_0, K_0) > 0$ and a unique solution in $\Gs$ on $[0, T_*)$ to the initial value problem associated to \eqref{abstract active scalar eqn MG}.

When $\alpha=-1$, there exist $T_*=T_*(\tau_0, K_0) > 0$ and a unique solution in $\Gs$ on $[0, T_*)$ to the initial value problem associated to \eqref{abstract active scalar eqn MG} for all $s\ge1$.
\end{thm}

When $\alpha\in[0,1]$, we claim that the conditions (C1)--(C6) listed in Section~\ref{illposed in Gevrey class general} hold for the symbol $\widehat{(-\Delta)^\frac{\alpha}{2} \mathbf{M}}$ when $a=1$. Then conditions (C1) and (C2) follow directly from the definition of $\widehat{\mathbf{M}}$. For the conditions (C3) and (C4), we define $\bbj$ by
\begin{align*}
\bbj:=(j^2,j),\qquad j\in\N.
\end{align*}
Then it is clear that $\bbj\in\Z^{2}$ and $|\bbj|=\sqrt{j^4+j^2}\to\infty$ as $j\to\infty$. And using \eqref{MG Fourier symbol_3}, for each $j\in\N$, we readily have
\begin{align*}
\widehat{(-\Delta)^\frac{\alpha}{2} \mathbf{M}_3}(\bbj,n)=\frac{(\beta^2/\eta)j^2(j^4+j^2)(j^4+j^2+n^2)^\frac{\alpha}{2}}{4\Omega^2n^2(j^4+j^2+n^2)+(\beta^2/\eta)^2j^4}\to0\mbox{ as $n\to\infty$},
\end{align*}
and
\begin{align*}
\widehat{(-\Delta)^\frac{\alpha}{2} \mathbf{M}_3}(\bbj,(n+1))<\widehat{(-\Delta)^\frac{\alpha}{2} \mathbf{M}_3}(\bbj,n)\mbox{ for all $n\in\mathbb{N}$},
\end{align*}
hence conditions (C3) and (C4) hold as well. Finally, to see why conditions (C5) and (C6) hold, for each $j$, $n\in\N$, there are positive constants $\tilde C_1$, $\tilde C_2$ which are independent of $j$ and $n$ such that 
\begin{align*}
\widehat{(-\Delta)^\frac{\alpha}{2} \mathbf{M}_3}(\bbj,n)\le\tilde C_1|\bbj|^3n^{\alpha-2},
\end{align*}
and
\begin{align*}
\widehat{(-\Delta)^\frac{\alpha}{2} \mathbf{M}_3}(\bbj,1)\widehat{(-\Delta)^\frac{\alpha}{2} \mathbf{M}_3}(\bbj,2)\ge\tilde C_2|\bbj|^{2\alpha+2}.
\end{align*}
If we choose $\beta_{1}=3$, $\beta_{2}=\alpha-2$ and $\beta_{3}=\alpha+1$, then by direct computation, for $\alpha\in[0,1]$, we can see that $\beta_{1}$, $\beta_{2}$ and $\beta_{3}$ satisfy \eqref{conditions on betas} with
\begin{align*}
\frac{\beta_{3}-\beta_{1}}{\beta_{3}\beta_{2}}=\frac{\alpha+1-3}{(\alpha+1)(\alpha-2)}=\frac{1}{\alpha+1}.
\end{align*}
Therefore, all the conditions (C1)--(C6) hold for $\widehat{(-\Delta)^\frac{\alpha}{2} \mathbf{M}}$ and hence we can apply Theorem~\ref{local ill-posed Gevrey thm general} to obtain the following result:

\begin{thm}\label{local ill-posed Gevrey MG thm}
Let $\alpha\in[0,1]$, $q>\alpha+\frac{7}{4}$ and $s>\frac{1}{\alpha+1}$. Then the equation \eqref{abstract active scalar eqn MG} is locally Lipschitz $(H^q , \Gs_{\tau})$ ill-posed for $\tau > 0$.
\end{thm}

\begin{rem}\label{discussion on sharp dichotomy}
Some remarks on Theorem~\ref{local existence Gevrey MG thm} and Theorem~\ref{local ill-posed Gevrey MG thm}:
\begin{itemize}
\item For $\alpha=0$, Theorem~\ref{local existence Gevrey MG thm} shows that the non-diffusive MG equation is locally well-posed in analytic space, while Theorem~\ref{local ill-posed Gevrey MG thm} shows that it is locally Lipschitz $(H^q , \Gs_{\tau})$ ill-posed for $\tau >0 $, $q>\frac{7}{4}$ and $s>1$. The results yield a more complete picture of the local-in-time well-posedness in Gevrey classes, which also strengthen some previous results on non-diffusive MG equation \cite{FS19, FV11b}. 
\item For $\alpha>0$, since $\frac{1}{\alpha+1}<1$, Theorem~\ref{local ill-posed Gevrey MG thm} implies that the equation \eqref{abstract active scalar eqn MG} is locally Lipschitz $(H^q , \Gs_{\tau})$ ill-posed for $\tau >0 $, $q>\alpha+\frac{7}{4}$ and $s\ge1$. In particular, it shows that for all $\alpha\in(0,1]$, the equation \eqref{abstract active scalar eqn MG} is not well-posed in the class of analytic functions.
\item The results obtained in Theorem~\ref{local existence Gevrey MG thm} and Theorem~\ref{local ill-posed Gevrey MG thm} give a sharp dichotomy across the value $\alpha = 0$. More precisely, for $\alpha<0$, the equations are locally well-posed in Gevrey spaces, while for $\alpha>0$, they are ill-posed in Gevrey spaces in the sense of Hadamard. Such dichotomy mainly comes from the transition of the order $\alpha+1$ singular operator $(-\Delta)^\frac{\alpha}{2}\mathbf{M}$, since in the spirit of the Cauchy-Kowalewskaya result \cite{CS00, GD10, KTVZ11}, it is possible to obtain local existence and uniqueness of solutions in spaces of real-analytic functions provided that the derivative loss in the nonlinearity $u\cdot\nabla \theta$ is of order at most one.
\end{itemize}
\end{rem}

\subsection{The singular incompressible porous media equation}\label{IPM subsection}

Next, we consider the following class of active scalar equations in $\mathbb{T}^2\times(0,\infty)$:
\begin{align}
\label{abstract active scalar eqn IPM}  
\left\{ \begin{array}{l}
\partial_t\theta+u\cdot\nabla\theta=0 \\
u=(-\Delta)^\frac{\alpha}{2}\mathbf{I}\theta,\,\,\,\theta(\bx,0)=\theta_0(\bx),
\end{array}\right.
\end{align}
where $\alpha\in[0,2)$. Here $\mathbf{I}=(\mathbf{I}_1,\mathbf{I}_2)$ is a Fourier multiplier operator with symbols $\mathbf{I}_1$, $\mathbf{I}_2$ given explicitly by \eqref{IPM Fourier symbol_1}-\eqref{IPM Fourier symbol_2}. For $\alpha\in(0,2]$, the system \eqref{abstract active scalar eqn IPM} is called the singular incompressible porous media (SIPM) equation. In particular, for $\alpha=0$, the system \eqref{abstract active scalar eqn IPM} reduces to the IPM equation given by \eqref{abstract active scalar eqn IPM intro}-\eqref{IPM Fourier symbol_2}. Furthermore, one can show that $u$ is divergence-free and $(-\Delta)^\frac{\alpha}{2} \mathbf{I}$ is an even, singular operator of order $\alpha$ with 
\begin{align*}
|\widehat{(-\Delta)^\frac{\alpha}{2}\mathbf{I}}(\bk)|\le C|\bk|^{\alpha},
\end{align*}
where $C$ is a fixed positive constant and is independent of $\alpha$ and $\bk$. As a direct consequence of Theorem~\ref{local existence Gevrey thm general}, we immediately obtain the following result for the case when $\alpha\in[0,1]$:

\begin{thm}\label{local existence Gevrey IPM thm}
For $\alpha\in(0,1]$, we fix $s\in[1,\frac{1}{\alpha}]$ and $K_0>0$. Let $\theta_0\in \Gs$ with radius of convergence $\tau_0>0$ and
\begin{align*}
\|(-\Delta)^\frac{r}{2}e^{\tau_0(-\Delta)^\frac{1}{2s}}\theta_0\|_{L^2}\le K_0,
\end{align*}
where $r>3$. There exist $T_*=T_*(\tau_0, K_0) > 0$ and a unique solution in $\Gs$ on $[0, T_*)$ to the initial value problem associated to \eqref{abstract active scalar eqn IPM}.

When $\alpha=0$, there exist $T_*=T_*(\tau_0, K_0) > 0$ and a unique solution in $\Gs$ on $[0, T_*)$ to the initial value problem associated to \eqref{abstract active scalar eqn IPM} for all $s\ge1$.
\end{thm}
When $\alpha\in[1,2)$, we claim that the conditions (C1)--(C6) listed in Section~\ref{illposed in Gevrey class general} hold for the symbol $(-\Delta)^\frac{\alpha}{2} \mathbf{I}$. For simplicity, we fix $a=1$. Then conditions (C1) and (C2) follow directly from the definition of $\widehat{\mathbf{I}}$. For the conditions (C3) and (C4), we define $\bbj$ by
\begin{align*}
\bbj:=j,\qquad j\in\N.
\end{align*}
Then it is clear that $\bbj\in\Z$ and $|\bbj|=j\to\infty$ as $j\to\infty$. And using \eqref{IPM Fourier symbol_2}, for each $j\in\N$, we readily have
\begin{align*}
\widehat{(-\Delta)^\frac{\alpha}{2} \mathbf{I}_2}(\bbj,n)=\frac{j^2}{(j^2+n^2)^\frac{2-\alpha}{2}}\to0\mbox{ as $n\to\infty$},
\end{align*}
and
\begin{align*}
\widehat{(-\Delta)^\frac{\alpha}{2} \mathbf{I}_2}(\bbj,(n+1))<\widehat{(-\Delta)^\frac{\alpha}{2} \mathbf{I}_2}(\bbj,n)\mbox{ for all $n\in\mathbb{N}$},
\end{align*}
hence conditions (C3) and (C4) hold as well. Finally, to see why conditions (C5) and (C6) hold, for each $j$, $n\in\N$, there are positive constants $\tilde C_1$, $\tilde C_2$ which are independent of $j$ and $n$ such that 
\begin{align*}
\widehat{(-\Delta)^\frac{\alpha}{2} \mathbf{I}_2}(\bbj,n)\le\tilde C_1|\bbj|^2n^{\alpha-2},
\end{align*}
and
\begin{align*}
\widehat{(-\Delta)^\frac{\alpha}{2} \mathbf{I}_2}(\bbj,1)\widehat{(-\Delta)^\frac{\alpha}{2} \mathbf{I}_2}(\bbj,2)\ge\tilde C_2|\bbj|^{2\alpha}.
\end{align*}
If we choose $\beta_{1}=2$, $\beta_{2}=\alpha-2$ and $\beta_{3}=\alpha$, then by direct computation, for $\alpha\in[1,2)$, we can see that $\beta_{1}$, $\beta_{2}$ and $\beta_{3}$ satisfy \eqref{conditions on betas} with
\begin{align*}
\frac{\beta_{3}-\beta_{1}}{\beta_{3}\beta_{2}}=\frac{\alpha-2}{(\alpha)(\alpha-2)}=\frac{1}{\alpha}.
\end{align*}
Therefore, all the conditions (C1)--(C6) hold for $\widehat{(-\Delta)^\frac{\alpha}{2} \mathbf{I}}$ and hence we can apply Theorem~\ref{local ill-posed Gevrey thm general} to obtain the following result:

\begin{thm}\label{local ill-posed Gevrey IPM thm}
Let $\alpha\in[1,2)$, $q>\alpha+\frac{1}{2}$ and $s>\frac{1}{\alpha}$. Then the equation \eqref{abstract active scalar eqn IPM} is locally Lipschitz $(H^q , \Gs_{\tau})$ ill-posed for $\tau > 0$.
\end{thm}

\begin{rem}
Some remarks on Theorem~\ref{local existence Gevrey IPM thm} and Theorem~\ref{local ill-posed Gevrey IPM thm}:
\begin{itemize}
\item For $\alpha=0$, Theorem~\ref{local existence Gevrey IPM thm} immediately implies that the IPM equation is locally well-posed in $G^s$ for all $s\ge1$.
\item For $\alpha=1$, Theorem~\ref{local existence Gevrey IPM thm} shows that the SIPM equation is locally well-posed in analytic space, while Theorem~\ref{local ill-posed Gevrey IPM thm} shows that it is locally Lipschitz $(H^q , \Gs_{\tau})$ ill-posed for $\tau >0 $, $q>\frac{3}{2}$ and $s>1$. 
\item For $\alpha\in(1,2)$, since $\frac{1}{\alpha}<1$, Theorem~\ref{local ill-posed Gevrey IPM thm} implies that the SIPM equation is locally Lipschitz $(H^q , \Gs_{\tau})$ ill-posed for $\tau >0 $, $q>\alpha+\frac{1}{2}$ and $s\ge1$. In particular, it shows that for all $\alpha\in(1,2)$, the equation is not well-posed in the class of analytic functions.
\item Similar to MG equations, the results obtained in Theorem~\ref{local existence Gevrey IPM thm} and Theorem~\ref{local ill-posed Gevrey IPM thm} give a sharp dichotomy across the value $\alpha = 1$ for the SIPM equation.
\end{itemize}
\end{rem}


\section*{Acknowledgment} S. Friedlander is supported by NSF DMS--1613135. A. Suen is supported by Hong Kong General Research Fund (GRF) grant project number 18300720, 18300821, 18300622 and Dean's Research Fund of the Faculty of Liberal Arts and Social Science, The Education University of Hong Kong, HKSAR, China (Project No. FLASS/DRF 04723). F. Wang is supported by the National Natural Science Foundation of China (No. 12101396 and 12161141004).

\bibliographystyle{amsalpha}

\bibliography{References_for_active_scalar}

\end{document}